\pdfoutput=1
\RequirePackage{ifpdf}
\ifpdf 
\documentclass[pdftex]{sigma}
\else
\documentclass{sigma}
\fi

\usepackage{enumitem}
\DeclareMathOperator{\Tr}{Tr}

\newtheorem{thm}{Theorem}[section]

 \newtheorem{cor}[thm]{Corollary}
 \newtheorem{lem}[thm]{Lemma}
 \newtheorem{prop}[thm]{Proposition}
{ \theoremstyle{definition} \newtheorem{defn}[thm]{Definition}
 \newtheorem{ex}[thm]{Example}
 \newtheorem{rmk}[thm]{Remark}}

\numberwithin{equation}{section}

\newcommand{\N}{\mathbb{N}}

\newcommand{\R}{\mathbb{R}}
\newcommand{\CC}{\mathbb{C}}

\newcommand{\cG}{{\mathcal G}}

\newcommand{\cI}{{\mathcal I}}

\newcommand{\Om}{\Omega}

\begin{document}

\newcommand{\arXivNumber}{1610.09620}

\renewcommand{\PaperNumber}{034}

\FirstPageHeading

\ShortArticleName{Results Concerning Almost Complex Structures on the Six-Sphere}

\ArticleName{Results Concerning Almost Complex Structures\\ on the Six-Sphere}

\Author{Scott O.~WILSON}

\AuthorNameForHeading{S.O.~Wilson}

\Address{Department of Mathematics, Queens College, City University of New York,\\ 65-30 Kissena Blvd., Queens, NY 11367, USA}
\Email{\href{mailto:scott.wilson@qc.cuny.edu}{scott.wilson@qc.cuny.edu}}
\URLaddress{\url{http://qcpages.qc.cuny.edu/~swilson/}}

\ArticleDates{Received November 20, 2017, in final form April 09, 2018; Published online April 17, 2018}

\vspace{-1mm}

\Abstract{For the standard metric on the six-dimensional sphere, with Levi-Civita connection $\nabla$, we show there is no almost complex structure $J$ such that $\nabla_X J$ and $\nabla_{JX} J$ commute for every $X$, nor is there any integrable $J$ such that $\nabla_{JX} J = J \nabla_X J$ for every $X$. The latter statement generalizes a previously known result on the non-existence of integrable ortho\-go\-nal almost complex structures on the six-sphere. Both statements have refined versions, expressed as intrinsic first order differential inequalities depending only on $J$ and the metric. The new techniques employed include an almost-complex analogue of the Gauss map, defined for any almost complex manifold in Euclidean space.}


\Classification{53C15; 32Q60; 53A07}

\vspace{-3mm}

\section{Introduction}

One knows from topology that the only spheres which admit almost complex structures are those in dimensions two and six, \cite{BS}. On the two-sphere, every almost complex structure is integrable, i.e., induced by a complex structure, and furthermore this is unique up to biholomorphism. The situation for the six-sphere $S^6$ is not well understood: while there are many almost complex structures, Hopf's original question from 1926, asking whether the six-sphere admits a complex structure, remains unsolved~\cite{H}. Thus it is of fundamental importance to better understand almost complex structures and concepts related to integrability, and in particular to determine whether $S^6$ has a complex structure.

There are results which forbid certain almost complex structures on $S^6$ from being integrable. Blanchard~\cite{Bl}, and then independently LeBrun~\cite{L}, have shown that no almost complex structure on $S^6$ which is orthogonal, with respect to the standard Euclidean metric on the sphere, is also integrable. Tang has shown in \cite{T} a more general analogous result for almost structures on $S^6$ that preserve any metric which satisfies a certain curvature-related positivity condition. Chern, in a result communicated by Bryant~\cite{Br}, has shown that no almost complex structure on $S^6$ that preserves a certain $2$-form is integrable.

In this paper, we proceed in the following way. Given an almost complex structure $J$ on a~manifold, and a connection $\nabla$ on the tangent bundle, one can ask that the covariant derivative~$\nabla_X J$ is $J$-linear in the variable~$X$, i.e., $\nabla_{J X } J = J \nabla_X J$. In fact, for any torsion free connection~$\nabla$, this condition \emph{implies} the integrability of $J$, so in this case we refer to $J$ as being strongly integrable with respect to~$\nabla$.

Inspired by LeBrun's work in~\cite{L}, we show using a variation of that argument that there are no almost complex structures on $S^6$ that are strongly integrable with respect to the Levi-Civita connection of the standard round metric. We recover the results of~\cite{Bl} and~\cite{L} as a special case since every integrable orthogonal almost complex structure on any Riemannian manifold is strongly integrable with respect to $\nabla$. On the other hand, the first order condition of $J$ being strongly integrable with respect to $\nabla$ is a priori weaker than the zeroth order condition of $J$ being orthogonal with respect to the metric.

We also provide a sharper result that prohibits further almost complex structures on $S^6$ from being integrable. This is deduced from an intrinsic first order differential inequality involving $J$ and the metric, see Corollary~\ref{cor;bound}.

Another result of this paper is to show that for any almost complex structure $J$ on $S^6$, integrable or not, there are non-trivial global conditions upon~$\nabla J$ that are dictated by topology. We define for any such $J$ a canonical map into a non-compact version of the Grassmannian, which supports a symplectic structure $\omega$ determined by the curvature of a tautological bundle with connection over this space. In fact, this map is defined for any almost complex manifold in Euclidean space, yielding a closed $2$-form associated to $J$ and the normal bundle. For the sphere we are able to calculate explicitly the pullback of this form, see Theorem~\ref{thm;Pcanon}. Several corollaries are deduced to obtain new non-trivial conditions on $\nabla J$. For example, there is no~$J$ on $S^6$ such that~$\nabla_X J$ and~$\nabla_{JX} J$ commute for all~$X$, see Corollary~\ref{cor;commute} and the sharper bounds thereafter. As a final corollary we also obtain from this viewpoint the prior result on the non-existence of orthogonal almost complex structures on~$S^6$, see Corollary~\ref{cor;L}.

The contents of this paper are as follows. In Section \ref{sec;int} we review notions of integrability, including a geometric characterization of integrability, that the Lie derivative $\mathcal{L}_X J$ of $J$ is $J$-linear in $X$. We study the aforementioned $J$-linearity of $\nabla_X J$ in the variable $X$, showing it is a strong notion of integrability, and giving both a real and an equivalent complexified formulation. Finally, we prove that on any Riemannian manifold, integrable orthogonal almost complex structures always satisfy this condition,
see Theorem \ref{thm;Jintorth}.

In Section \ref{sec;grass} we introduce models for the Grassmannian and its non-compact version, to be used in later sections. While one can make homogeneous space or vector sub-space definitions, we prefer to present them as spaces of idempotent matrices since this makes the canonical maps defined in later sections most transparent. We verify that this non-compact version of the Grassmannian also carries a symplectic structure, with explicit symplectic form tamed by many almost complex structures. On the compact Grassmannian subspace of self-adjoint idempotent matrices, we obtain a simple formulation of the K\"ahler structure in terms of the standard matrix inner product $\Tr (AB^*)$.

In Section \ref{sec;Pmap} we introduce a generalization of the Gauss map, defined for an almost complex manifold in Euclidean space, which yields a canonical map into the non-compact Grassmannian. We use the term ``canonical'' since the map is induced by the projection operator $P^- = \frac 1 2 (I + i J)$ and projection onto the complexified normal bundle. For the standard $S^6$ and any $J$, we calculate the pullback of the symplectic form on the target explicitly in terms of the metric and $\nabla J$, and the corollaries mentioned above are deduced.

In Section \ref{sec;Gmap} we define for any $J$ on $S^6$ an immersion of $S^6$ into the Grassmannian of self-adjoint idempotent matrices. This is similar to the map used in \cite{L}, though modified to include the normal bundle of the sphere in Euclidean space. We calculate $\overline{\partial}$ of this map and show that it vanishes if and only if $\nabla_X J$ is $J$-linear in the variable $X$. We conclude there are no such strongly integrable almost complex structures on $S^6$, all of which would have necessarily been integrable.
Finally, we compute the pullback of the K\"ahler form explicitly and, as a consequence, we deduce an intrinsic first order differential inequality depending only on $J$ and the metric which, if satisfied everywhere, guarantees $J$ to be non-integrable.

\section{Integrability} \label{sec;int}

An almost complex structure on a smooth manifold $M$ is a section $J$ of the endomorphism bundle $\operatorname{End}(TM)$ which squares to $-\operatorname{Id}$. Such a $J$ is said to be integrable if it is everywhere induced by the complex structure $i$ on $\CC^n$ along coordinate charts. By a theorem of Newlander--Nirenburg~\cite{NN}, $J$ is integrable if and only if the Nijenhuis tensor
\begin{gather} \label{eq:N}
N(X,Y) = J[X,JY] + J[JX,Y] + [X,Y] - [JX,JY]
\end{gather}
vanishes everywhere. A more geometric way to understand this condition is given by the following lemma.

\begin{lem}An almost complex structure $J$ is integrable if and only if the Lie derivative of $J$ is $J$-linear in the direction of any vector field $X$, i.e.,
\begin{gather*}
\mathcal{L}_{JX} J = J \mathcal{L}_{X} J
\end{gather*}
for all vector fields $X$.
\end{lem}

\begin{proof}The condition given holds if and only if $( \mathcal{L}_{JX} J ) Y= ( J \mathcal{L}_{X} J ) Y$ for all vectors $Y$, which is equivalent to equation~\eqref{eq:N} using the identity
\begin{gather*}
( \mathcal{L}_{X} J ) Y = \mathcal{L}_{X} (JY) - J ( \mathcal{L}_{X} Y ) = [X,JY] - J[X,Y].\tag*{\qed}
\end{gather*} \renewcommand{\qed}{}
\end{proof}

A connection $\nabla$ on $TM$ induces, by the Leibniz property, a connection on $\operatorname{End}(TM)$, which we also denote by $\nabla$. In particular, $J^2=-\operatorname{Id}$ implies
\begin{gather*}
( \nabla_X J) J = -J (\nabla_X J)
\end{gather*}
for all vectors $X$, so that $\nabla J$ is $J$-anti-linear in its second argument.

This paper concerns the $J$-linearity of $\nabla J$ in its first argument, which is stronger than integrability when $\nabla$ is torsion free, cf.\ Lemma~\ref{lem;strongimpliesint}. Recall that a connection $\nabla$ on the tangent bundle of a manifold~$M$ is said to be torsion free if
 \begin{gather*}
 \nabla_X Y - \nabla_Y X = [X,Y]
 \end{gather*}
 for all vector fields $X$ and $Y$.

\begin{defn} \label{defn:strongint}
Let $\nabla$ be a torsion free connection on the tangent bundle of an almost complex manifold $(M,J)$. We say $J$ is \emph{strongly integrable with respect to $\nabla$} if $\nabla_X J$ is $J$-linear in the variable~$X$, i.e.,
\begin{gather*}
\nabla_{JX} J = J \nabla_X J
\end{gather*}
for all vectors $X$.
\end{defn}

\begin{lem} \label{lem;strongimpliesint}
If $J$ is strongly integrable with respect to some torsion free connection $\nabla$, then $J$ is integrable.
\end{lem}

\begin{proof}
If $\nabla$ is torsion free, then using the identity $(\nabla_X J) Y = \nabla_X (J Y) - J \nabla_X Y$, we have
\begin{gather}
N(X,Y) = (\nabla_{JY} J - J \nabla_Y J) X - (\nabla_{JX} J - J \nabla_X J)Y, \label{eq:Nnabla}
\end{gather}
and that proves the lemma.
\end{proof}

In two real dimensions every almost complex structure is integrable, so the next example shows that an integrable $J$ need not be strongly integrable with respect to every connection. At the end of the section we will conclude that any integrable $J$ is strongly integrable with respect to some torsion free connection.

\begin{ex} \label{ex;2dex}Let $M = \{ (x,y) \in \R \times \R\, | \, x \neq 0 \}$ and let
$J$ on $M$ be given by
\begin{gather*}
J(x,y) =
\begin{bmatrix}
0 & x \\
\frac {-1}{ x} & 0
\end{bmatrix}.
\end{gather*}
With the standard metric and Levi-Civita connection $\nabla$ we have $\nabla_{\frac{\partial }{\partial y}} J(x,y) = 0$
and
\begin{gather*}
\nabla_{\frac{\partial }{\partial x}} J(x,y) =
\begin{bmatrix}
0 & 1 \\
\frac{1}{ x^2} & 0
\end{bmatrix}.
\end{gather*}
So, $J$ is not strongly integrable with respect to $\nabla$ since
\begin{gather*}
 \nabla_{ J \frac{\partial }{\partial y}} J(x,y)= \nabla_{x \frac{\partial }{\partial x}} J(x,y) =
\begin{bmatrix}
0 & x \\
\frac{1}{ x} & 0
\end{bmatrix}, \\
J \nabla_{\frac{\partial }{\partial y}} J(x,y)= 0 .
\end{gather*}
\end{ex}

The proof of Lemma \ref{lem;strongimpliesint} suggests considering, for any torsion free connection $\nabla$, the deviation from $\nabla_X J$ being $J$-linear in $X$, i.e.,
\begin{gather} \label{eq:tanalg}
m(X,Y) := ( \nabla_{JX} J - J \nabla_X J ) Y.
\end{gather}
It is straightforward to check that $m(X,Y)$ is a tensor (i.e., linear over functions in each variable), and that
\begin{gather} \label{eq:misJantilinear}
m(JX,Y) = m(X,JY) = - J m(X,Y),
\end{gather}
so that $m$ is $J$ anti-linear. One can regard $m$ as a smoothly varying family of bi-linear operations on the tangent spaces, so we will refer to this as the \emph{tangent algebra} associated to $J$ and $\nabla$.
We can recast the previous lemmas as

\begin{cor} \label{cor:Jintcomm}
 An almost complex structure is integrable if and only if, for any torsion free connection, the tangent algebra associated to $J$ and the connection is \emph{everywhere commutative}. An almost complex structure is strongly integrable with respect to a torsion free connection if and only if the associated tangent algebra is \emph{everywhere vanishing}.
 \end{cor}

 \begin{proof}Using the definition in equation~\eqref{eq:tanalg}, the first statement follows from equation~\eqref{eq:Nnabla} and the second
statement follows from Definition (\ref{defn:strongint}).
 \end{proof}

\begin{rmk}A tangent algebra on any surface defines (at any point) an operation which is either zero or a division algebra. In fact, one can see from equation~\eqref{eq:misJantilinear} that the tangent algebra is uniquely determined at any point by its value $m(X,X)$ for some single vector $X$. In higher dimensions this is far from the case, as we have direct sums of (trivial and non-trivial) rank two examples, such as on the vector space $\CC$ with product $(a,b) \mapsto \epsilon \overline{ab}$, for any $\epsilon \in \R$. It is perhaps worth noting that there are no division algebras in dimension six.
\end{rmk}

Recall that given $(M,J)$ we define the holomorphic and anti-holomorphic tangent spaces to be the $+i$ and $-i$ eigenspaces of $J$ on $TM \otimes \CC$, i.e.,
\begin{gather*}
T^{1,0}_x = \{ v \in T_xM \otimes \CC \,| \, Jv = i v\}, \\
T^{0,1}_x = \{ v \in T_xM \otimes \CC\, | \, Jv = - i v\}.
\end{gather*}
The following lemma shows that strong integrability is equivalent to the connection respecting the $T^{0,1}$ summand.

\begin{lem} \label{lem:strongintequiv}
Let $(M,J)$ be an almost complex manifold and $\nabla$ any torsion free connection. The following are equivalent:
\begin{itemize}\itemsep=0pt
\item $J$ is strongly integrable with respect to $\nabla$, i.e., the tangent algebra vanishes.
\item The $\CC$-linear extension of $\nabla$ satisfies $\nabla\colon T^{0,1} \otimes T^{0,1} \to T^{0,1}$.
\end{itemize}
\end{lem}

The second condition should be compared with the integrability condition that $T^{0,1}$ is closed under the Lie bracket, thus giving another verification of Lemma~\ref{lem;strongimpliesint}.

\begin{proof}Any $V,W \in T^{0,1}$ can be written uniquely as $V = X + iJX$ and $W = Y + iJY$ for some real $X$ and $Y$. Then
\begin{gather*}
\nabla_{X+ iJX} (Y + i JY) = (\nabla_X Y - \nabla_{JX} (JY) ) + i (\nabla_X (JY) + \nabla_{JX} Y ).
\end{gather*}
The right hand side is in $ T^{0,1}$ if and only if
\begin{gather*}
J (\nabla_X Y - \nabla_{JX} (JY) ) = (\nabla_X (JY) + \nabla_{JX} Y ).
\end{gather*}
Using the identity $\nabla_{JX} Y + J\nabla_{JX} (JY) = J ( \nabla_{JX} J ) Y$, this is equivalent to
\begin{gather*}
Jm(X,Y)= ( \nabla_X J +J ( \nabla_{JX} J ) ) Y = 0
\end{gather*}
for all $X$ and $Y$, which holds for all $X$ and $Y$ if and only if $J$ is strongly integrable with respect to $\nabla$.
\end{proof}

In fact, for orthogonal almost complex structures, the two conditions of the previous lemma are equivalent to integrability (see \cite[p.~267]{S}). For completeness here, we give our own proof in one direction, inspired by an argument in \cite[Appendix~C.7]{McS}. There a stronger hypothesis (namely, that the form $\omega$ below is also closed) is used to conclude the stronger consequence (that $\nabla J \equiv 0$).

\begin{thm} \label{thm;Jintorth} Let $M$ be a smooth manifold with Riemannian metric $\langle - ,- \rangle$, Levi-Civita connection $\nabla$, and integrable almost complex structure $J$. If $J$ is compatible with the metric, i.e., $\langle JX,JY \rangle = \langle X,Y \rangle $ for all $X,Y$, then $J$ is strongly integrable with respect to $\nabla$.
\end{thm}

\begin{proof}Since $J$ is orthogonal, there is an induced $2$-form $\omega$ defined by
\begin{gather*}
\omega(X,Y) = \langle JX, Y \rangle.
\end{gather*}
For any given vectors $X$, $Y$ and $Z$ at a fixed point $p \in M$, choose extensions of these vectors to vector fields so that
 all six covariant derivatives of the form $\nabla_V W$ vanish at the point $p$. Then, at the point~$p$, all Lie brackets also vanish, and using the formula for $d$ in terms of the Lie bracket, and the fact that $\nabla$ respects the metric, we have
\begin{gather*}
d \omega (X,Y,Z) = \langle ( \nabla_X J ) Y, Z\rangle + \langle ( \nabla_Y J ) Z, X \rangle + \langle ( \nabla_Z J ) X, Y\rangle .
\end{gather*}
Using equation~\eqref{eq:Nnabla} and the previous equation we calculate
\begin{gather*}
\langle X, N(Y,Z) \rangle =
\langle X, ( \nabla_{JZ} J) Y\rangle - \langle X, J ( \nabla_Z J ) Y\rangle
 - \langle X, ( \nabla_{JY} J) Z\rangle
 + \langle X, J( \nabla_Y J) Z\rangle \\
\hphantom{\langle X, N(Y,Z) \rangle}{} = - \langle X, J ( \nabla_Z J ) Y\rangle + \langle X, J ( \nabla_Y J ) Z\rangle
 - \langle JZ, ( \nabla_{Y} J ) X\rangle + \langle Z, ( \nabla_{X} J ) (JY)\rangle \\
\hphantom{\langle X, N(Y,Z) \rangle=}{}
 - \langle Y, ( \nabla_{X} J ) (JZ)\rangle + \langle JY, ( \nabla_{Z} J ) X\rangle
 + d \omega (X,JZ,Y) - d \omega (X,JY,Z)\\
\hphantom{\langle X, N(Y,Z) \rangle}{}
= - 2 \langle J ( \nabla_X J ) Y, Z \rangle + d \omega (X,JZ,Y) - d \omega (X,JY,Z),
\end{gather*}
where in the last equality we have used the fact that $J$ and $\nabla J$ are both skew-adjoint. So,
\begin{gather*}
2 \langle J ( \nabla_X J ) Y, Z \rangle = -\langle X, N(Y,Z) \rangle + d \omega (X,JZ,Y) - d \omega (X,JY,Z), \\
2 \langle ( \nabla_{JX} J ) Y, Z \rangle = 2 \langle J ( \nabla_{JX} J ) Y, JZ \rangle \\
\hphantom{2 \langle ( \nabla_{JX} J ) Y, Z \rangle}{}
= -\langle JX, N(Y,JZ) \rangle + d \omega (JX,J^2Z,Y) - d \omega (JX,JY,JZ) \\
\hphantom{2 \langle ( \nabla_{JX} J ) Y, Z \rangle}{} = \langle X, N(Y,Z) \rangle - d \omega (JX,Z,Y) - d \omega (JX,JY,JZ).
\end{gather*}
Subtracting these equations we have
\begin{gather*}
2\langle ( J ( \nabla_X J ) -( \nabla_{JX} J ) ) Y , Z \rangle= -2 \langle X, N(Y,Z) \rangle\\
\qquad{} + d \omega (X,JZ,Y) - d \omega (X,JY,Z) + d \omega (JX,Z,Y) + d \omega (JX,JY,JZ).
\end{gather*}
Using $N \equiv 0$ we have
\begin{gather*}
2\langle ( ( \nabla_{JX} J ) - J ( \nabla_X J ) ) Y , Z \rangle\\
\qquad{} = d \omega (X,Y, JZ) + d \omega (X,JY,Z) + d \omega (JX,Y,Z) - d \omega (JX,JY,JZ).
\end{gather*}
The right hand side of this equation is alternating in $X$ and~$Y$, but by Corollary~\ref{cor:Jintcomm}, the left hand side is symmetric in~$X$ and~$Y$, so both sides vanish for all $Z$, and therefore $ ( ( \nabla_{JX} J ) - J ( \nabla_X J ) )$ $\equiv 0$.
\end{proof}

The argument also shows that, under the hypotheses of the theorem,
\begin{gather*}
d \omega (JX,JY,JZ) = d \omega (JX,Y,Z) + d \omega (X,JY,Z) + d \omega (X,Y, JZ).
\end{gather*}

\begin{rmk} \label{rmk;orthversusstrong}
According to LeBrun \cite{L}, if $J$ on $S^{2n}$ is integrable and orthogonal with respect to the standard Euclidean metric, then for $V,W \in T^{0,1}$ and the Levi-Civita connection $\nabla$, we have
\begin{gather*}
\nabla_V W = \frac 1 2 [V,W] \in T^{0,1},
\end{gather*}
where the containment follows from integrability. The previous theorem implies in this case that~$J$ is strongly integrable with respect to~$\nabla$, (as also does Lemma~\ref{lem:strongintequiv}). Note that $\nabla_V W \in T^{0,1}$ is a priori weaker than the condition $\nabla_V W = \frac 1 2 [V,W] \in T^{0,1}$.
 \end{rmk}

\begin{cor}If $J$ is integrable then there exists a torsion free metric preserving connection~$\nabla$ such that $J$ is strongly integrable with respect to~$\nabla$.
\end{cor}

To contrast with the K\"ahler case, of course there need not exist a torsion free metric preserving connection $\nabla$ such that $\nabla J = 0$.

\begin{proof}Choose any metric compatible with $J$ (for example, one may average any given metric with respect to $J$) and then let $\nabla$ be the Levi-Civita connection associated to this metric.
\end{proof}

\section{Grassmannian of idempotent matrices} \label{sec;grass}

We begin with the definitions of the Grassmannian and its non-compact version, both presented as a space of idempotent matrices (see, e.g., \cite[Problem~5-C]{MS}).

\begin{defn}For $n \in \N$ and $0 \leq k \leq n$ let
\begin{gather*}
\cI_{k,n} = \big\{ P \in \mathcal{M}_{n \times n}(\CC) \,| \, P^2 = P, \, \operatorname{rank}(P) = k \big\}
\end{gather*}
be the space of idempotent $n$ by $n$ complex matrices of rank $k$, and let $ \cG_{k,n} \subset \cI_{k,n}$ be the subspace of self-adjoint matrices, given by
\begin{gather*}
\cG_{k,n} = \big\{ P \in \mathcal{M}_{n \times n}(\CC)\, | \, P^2 = P ,\, \operatorname{rank}(P) = k, \, P^* = P \big\},
\end{gather*}
where the adjoint is taken with respect to the standard complex inner product on $\CC^n$.
\end{defn}

Note that $P \in \cG_{k,n}$ if and only if $P$ is an orthogonal projection of rank~$k$. So, the latter space $\cG_{k,n}$ can be identified with the Grassmannian of $k$-planes in $\CC^n$, since there is a~unique orthogonal projection whose image is a given $k$-dimensional subspace. The space $\cI_{k,n}$ can be described as the space of two complementary subspaces of~$\CC^n$, of dimensions~$k$ and $n-k$.

The relation $P^2 = P$ implies, for any path $P_t \in \cI_{k,n}$, that
\begin{gather*}
P_t' P_t + P_t P'_t = P'_t,
\end{gather*}
so that $(I-P_t)P_t' = P_t'P_t$ and $P_t P_t' = P_t' (I-P_t)$. For short we may write this as $dPP= (I-P)dP$ and $P dP= dP (I-P)$. The tangent spaces are
\begin{gather*}
T_P\cI_{k,n} = \big\{ X \in \mathcal{M}_{n \times n}(\CC)\, |\, XP= (I-P)X \, \textrm{and} \, PX= X(I-P) \big\}
\end{gather*}
and
\begin{gather*}
T_P \cG_{k,n} = \big\{ X \in \mathcal{M}_{n \times n}(\CC) \,| \, XP = P^\perp X, \, PX= X P^\perp , \, \textrm{and} \, X^*=X \big\},
\end{gather*}
of dimensions $2k(n-k)$, and $k(n-k)$, respectively.

Notice that $P =\left[\begin{smallmatrix}
1 & 0 \\
0 & 0
\end{smallmatrix}\right]$ in any basis respecting the splitting $\CC^n = \operatorname{Im}(P) \oplus \operatorname{Ker}(P)$, so that $X \in T_P \cI_{k,n}$ if and only if
$X =\left[\begin{smallmatrix}
0 & B \\
C & 0
\end{smallmatrix}\right]$ for some linear operators $B\colon \operatorname{Ker}(P) \to \operatorname{Im}(P)$ and $C \colon \operatorname{Im}(P) \to \operatorname{Ker}(P)$. If $X \in T_P \cG_{k,n}$ then $\operatorname{Im}(P)$ and $\operatorname{Ker}(P)$ are orthogonal and we have $C = B^*$.

The subspaces $\operatorname{Im}(P)$ and $\operatorname{Ker}(P)$ form the fibers of two vector bundles $\operatorname{Im} \to \cI_{k,n}$ and
$\operatorname{Ker} \to \cI_{k,n}$, of ranks $k$ and $n-k$, respectively. By the above, the tangent bundle of $\cI_{k,n}$ is simply the bundle $\operatorname{Hom}( \operatorname{Ker}, \operatorname{Im}) \oplus \operatorname{Hom}( \operatorname{Im}, \operatorname{Ker})$, and the tangent bundle of $\cG_{k,n}$ is isomorphic to $\operatorname{Hom}(\operatorname{Im},\operatorname{Ker})$, as expected. Note all of these bundles are contained inside a trivial complex bundle.

In particular, the bundle $\operatorname{Im} \to \cI_{k,n}$ is a sub-bundle of the trivial $\CC^n$ bundle, so it inherits a~connection given by
the trivial connection~$d$ on the trivial~$\CC^n$ bundle, followed (at each point~$P$) by the projection $P$ onto $\operatorname{Im}(P) \subset \CC^n$. In short, $\nabla^{\operatorname{Im}} = P \circ d$. The curvature of this connection (cf.~\cite[p.~344]{K}) is given by the matrix-valued 2-form
\begin{gather*}
R = P \wedge dP \wedge dP,
\end{gather*}
which we may write for short as $PdP^2$.

We will be interested in the closed $\CC$-valued $2$-form\footnote{It is customary to include a factor of $\pi$ in the denominator so that the resulting class is integral. In order to simplify the presentation we drop this factor since it has no impact on the results here.}
\begin{gather*}
\omega = \frac {1}{2i} \Tr\big(P dP^2\big) \in \Om^2 ( \cI_{k,n} ; \CC).
\end{gather*}
Concretely, for $X,Y \in T_P \cI_{k,n}$ given as above by
$X =\left[
\begin{smallmatrix}
0 & B \\
C & 0
\end{smallmatrix}\right]$ and
$Y =\left[
\begin{smallmatrix}
0 & D \\
E & 0
\end{smallmatrix}\right]$, we have
\begin{gather*}
\omega(X,Y) = \frac {1}{2i} \Tr(PXY - PYX) = \frac {1}{2i} \Tr (BE - DC) .
\end{gather*}
Note that for $X,Y \in T_P \cG_{k,n}$ we have
\begin{gather*}
\omega(X,Y) = \frac {1}{2i} \Tr(BD^* - DB^*) = \operatorname{Im} \Tr(BD^*).
\end{gather*}
Now it is clear that $\omega$ is a K\"ahler form on $ \cG_{k,n}$ since for any $X \in T_P \cG_{k,n}$ given by
$X =\left[
\begin{smallmatrix}
0 & B \\
B^* & 0
\end{smallmatrix}\right]$, and the $J$ operator on $T_P \cG_{k,n}$ defining the complex structure on $\cG_{k,n}$ given explicitly by
\begin{gather*}
J\left(
\begin{bmatrix}
0 & A \\
A^* & 0
\end{bmatrix}
\right)
=
\begin{bmatrix}
0 & -iA \\
iA^* & 0
\end{bmatrix},
\end{gather*}
we have
\begin{gather*}
\omega(X,JY) = \operatorname{Im} \Tr \big(B (-iD)^*\big) = \operatorname{Re} \Tr(B D^*),
\end{gather*}
which is a real inner product on $T_P \cG_{k,n}$.

For completeness we record that $\cI_{k,n}$ is a symplectic manifold, whose symplectic form $\operatorname{Re}(\omega)$ is tamed by many (possibly non-integrable) almost complex structures $J$, which by definition means that $\operatorname{Re}(\omega)(X,JX) >0$ for all $X\neq 0$. The form will be used in subsequent sections, though the operators $J$ considered in the proof will not.

\begin{thm}For each $0 < k < n$, the closed real valued $2$-form $\operatorname{Re}(\omega)$ is a symplectic form on the manifold $\cI_{k,n}$ whose restriction to $\cG_{k,n}$ is the K\"ahler form $\omega$.
\end{thm}

\begin{proof} For any choice of inner product on $\CC^n$, the $2$-form $\operatorname{Re}(\omega)$ is tamed by the almost complex structure on $\cI_{k,n}$ given by
\begin{gather*}
JX = J\left(
\begin{bmatrix}
0 & B \\
C & 0
\end{bmatrix}
\right)
=
\begin{bmatrix}
0 & -iC^* \\
iB^* & 0
\end{bmatrix},
\end{gather*}
since for $X\neq 0$
\begin{gather*}
\operatorname{Re}(\omega)(X,JX) = \operatorname{Re} \left ( \frac {1}{2i} \Tr ( B(iB^*) - (-iC^*)C) \right) = \frac 1 2 \Tr(BB^* + C^*C) > 0.\tag*{\qed}
\end{gather*} \renewcommand{\qed}{}
\end{proof}

The imaginary part of the $2$-form $\omega$, defined on $\cI_{k,n}$, is d-exact. To see this, consider the retraction $\cI_{k,n} \to \cG_{k,n}$ which assigns to a pair of complementary planes the unique orthogonal pair which keeps the first plane fixed. This is a fibration with contractible fibers, therefore a~cohomology isomorphism, and the restriction of the complex form $\omega$ to $\cG_{k,n}$ is real.

\section{The canonical map} \label{sec;Pmap}

\begin{defn}
Let $M$ be a $2k$-dimensional submanifold of $\R^n$, and let $J$ be an almost complex structure on $M$. At any point $x \in M$ we have
\begin{gather*}
\CC^n = \R^n \otimes \CC = ( T_x M \oplus N_x ) \otimes \CC = \big( T^{0,1}_x \oplus ( N_x \otimes \CC ) \big) \oplus T^{1,0}_x,
\end{gather*}
where $N_x$ is the fiber at $x$ of the normal bundle $N$ of $M$ with respect to the standard Euclidean metric.

The \emph{canonical map} $P\colon M \to \cI_{n-k,n}$ is defined for $x \in M$ so that $P(x)$ is the unique projection on $\CC^n$ with image equal to $\big( T^{0,1}_x \oplus ( N_x \otimes \CC) \big)$ and kernel equal to $T^{1,0}_x$. We'll write
\begin{gather*}
P = P^- \oplus \pi_{N_x \otimes \CC},
\end{gather*}
where $\pi_{N_x \otimes \CC}\colon \CC^n \to \CC^n$ is orthogonal projection onto $N_x \otimes \CC$.
\end{defn}

First we remark that we are using the non-compact Grassmannian in this section to obtain Theorem~\ref{thm;Pcanon} and its first two corollaries, which concern all almost complex structures. In the next section, for results concerning the non-existence of strongly integrable almost complex structures, it suffices to work with the compact Grassmannian.

Next we note there is obviously some choice here in defining the canonical map, where we have chosen the map to have image $\operatorname{Im}(P^-) \oplus ( N \otimes \CC )$, versus $\operatorname{Im}(P^+) \oplus ( N \otimes \CC )$, or $\operatorname{Im}(P^\pm)$. The choice of $\operatorname{Im}(P^-)= T^{0,1}$ is motivated by Lemma~\ref{lem:strongintequiv}, and we include the normal bundle since the covariant derivative on the sphere with standard metric has a normal component, as given in equation~\eqref{eq;NablaSphere} below.

\begin{rmk}The decomposition $TM \otimes \CC = \operatorname{Im}(P^+) \oplus \operatorname{Im}(P^-)$ is an orthogonal direct sum decomposition if and only if $J$ is orthogonal. It follows that the canonical map $P\colon M \to \cI_{n-k,n}$ factors through $\cG_{n-k,n}$ if and only if $J$ is orthogonal.
\end{rmk}

\begin{lem} \label{lem;Pcanonemb}
Let $n=1$ or $n=3$. For any $J$ on $S^{2n} \subset \R^{2n+1}$ the induced canonical map $P\colon S^{2n} \to \cI_{n+1,2n+1}$ is an immersion.
\end{lem}

\begin{proof} The idea is to use the geometric fact that normal direction of the sphere moves tangentially. Explicitly, if $X \in T_xS^{2n}$ then $d_X P \in T_{P(x)} \cI_{n+1,2n+1}$ is a non-zero matrix since, for the real normal unit vector $n_x$ to $S^{2n}$ at $x$, we have
\begin{gather*}
(d_XP)(n_x) = d_X(P(n_x)) - P d_X(n_x) = (I-P) d_X(n_x) = P^+(X) = \frac 1 2 (X - i J X).
\end{gather*}
So, $d_X P$ is non-zero whenever $X$ is non-zero.
\end{proof}

Before stating the main result of this section, we make the following remarks. An almost complex structure $J$ satisfies $\Tr(J)=0$ at every point, $\Tr(\nabla_X J) =0$ for every tangent vector $X$, and of course, $\Tr ([\nabla_X J, \nabla_Y J] )=0$ for any two tangent vectors $X$ and $Y$, where $[-,-]$ denotes the commutator. These also hold for the $\CC$-linear extension of these operators to $TM \otimes \CC$.

 Now, for any $X$ and $Y$, the commutator $[\nabla_X J, \nabla_Y J]$ preserves the splitting $TM \otimes \CC = T^{1,0} \oplus T^{0,1}$ since it commutes with $J$. The statement below pertains to the restriction of $[\nabla_X J, \nabla_Y J]$ to $T^{0,1}$ which is denoted by
 \begin{gather*}
[ \nabla_X J , \nabla_Y J ] \big|_{T^{0,1}} \colon \ T^{0,1} \to T^{0,1}.
 \end{gather*}
 Note that $ [ \nabla_X J , \nabla_Y J ] \big|_{T^{0,1}}$ is the difference of two compositions $T^{0,1} \to T^{1,0} \to T^{0,1}$, in opposite order, so the trace need not be zero, and is a priori an arbitrary complex number.

Recall the form $\operatorname{Re}(\omega) = \operatorname{Re}\big( \frac{1}{ 2i} \Tr\big(PdP^2\big) \big) \in \Omega^2(\cI_{n+1,2n+1};\CC)$ from Section~\ref{sec;grass}, and the cano\-ni\-cal map $P\colon S^{2n} \to \cI_{n+1,2n+1}$ induced by an arbitrary~$J$.

\begin{thm} \label{thm;Pcanon}
Let $\langle -,- \rangle $ be the Euclidean metric of $\R^{2n+1}$ where $n=1$ or $n=3$. Let $J$ be any almost complex structure on the unit sphere $S^{2n} \subset \R^{2n+1}$, with induced metric and Levi-Civita connection $\nabla$. Let $P\colon S^{2n} \to \cI_{n+1,2n+1}$ be the induced canonical map. The pullback of $\operatorname{Re}(\omega)$ along $P$ is given by
\begin{gather*}
 P^* \operatorname{Re} ( \omega) (X,Y) =-\frac{1}{8} \operatorname{Im} \Tr \big([ \nabla_X J , \nabla_Y J ] \big|_{T^{0,1}} \big)
 - \frac 1 4 \langle X,JY \rangle + \frac 1 4 \langle Y,JX \rangle .
\end{gather*}
\end{thm}

\begin{proof} Let $X$ and $Y$ be fixed tangent vectors at a point $x$. As in Section~\ref{sec;grass}, in any complex basis of $(T_xM \otimes \CC) \oplus (N_x \otimes \CC)$ for which the projection $P(x)$ is of the form
$P(x) = \left[\begin{smallmatrix}
1 & 0 \\
0& 0
\end{smallmatrix}\right]$ we have
$d_X P =\left[
\begin{smallmatrix}
0 & B \\
C& 0
\end{smallmatrix}\right] $ for complex matrices $B\colon T^{1,0}_x \to T^{0,1}_x \oplus ( N_x \otimes \CC )$ and $C \colon T^{0,1}_x \oplus ( N_x \otimes \CC ) \to T^{1,0}_x$, which are the restriction of $d_X P$ to their appropriate domains.

On the sphere with the standard Euclidean metric we have
\begin{gather} \label{eq;NablaSphere}
d_X W = \nabla_X W - \langle X,W \rangle n_x,
\end{gather}
where $n_x$ is the outward unit normal vector at $x$. This equation also holds for complex vector fields $W$ by extending the structures linearly over $i \in \CC$. For the remainder of this proof (only), we will reserve the notation $ \langle - , - \rangle$ for the extension of the standard metric in $\R^{2n+1}$ to a~bilinear map which is $\CC$-linear \emph{in both entries}.

We first show that
\begin{gather*}
B= \frac i 2 (\nabla_X J) + \langle X , \, -\, \rangle n_x,
\end{gather*}
which is not surprising since $P = \frac 1 2 (I + i J) \oplus \pi_{N_x \otimes \CC}$.
Explicitly, for any $Z - i JZ \in T^{1,0} $ we have
\begin{gather*}
B(Z-iJZ) = (d_XP)(Z-iJZ) = d_X(P(Z - iJZ)) - P(d_X(Z-iJZ)) \\
\hphantom{B(Z-iJZ)}{} = - P(d_X(Z-iJZ)) = -P ( \nabla_X (Z-iJZ) - \langle X,Z - iJZ \rangle n_x )\\
\hphantom{B(Z-iJZ)}{}= - P^- \nabla_X (Z-iJZ) + \langle X,Z - iJZ \rangle n_x ,
\end{gather*}
where the last equality follows from the definition $P\big|_{TM \otimes \CC} = P^- $ and $P(n_x) = n_x$. Now using $P^- = \frac 1 2 (I + i J)$ we have
\begin{gather*}
 -P^- \nabla_X (Z-iJZ) = -P^- \nabla_X (P^+ Z) = \nabla_X ( -P^- P^+ Z) + \nabla_X (P^-) (P^+ Z) \\
 \hphantom{-P^- \nabla_X (Z-iJZ)}{} =\frac 1 2 \nabla_X (I +iJ) (Z-iJZ) = \frac i 2 (\nabla_X J )(Z- i J Z),
\end{gather*}
which completes the proof that $B= \frac i 2 (\nabla_X J) + \langle X , \, -\, \rangle n_x$.

Also, $C \big|_{T^{0,1}} = \frac{i}{ 2} (\nabla_X J) $ since
\begin{gather*}
C(Z+iJZ) = (d_X P)(Z+iJZ) = d_X(P(Z + iJZ)) - P(d_X(Z+iJZ)) \\
\hphantom{C(Z+iJZ)}{} =d_X(Z + iJZ) - P(d_X(Z+iJZ)) = (I-P)d_X (Z+iJZ) \\
\hphantom{C(Z+iJZ)}{} = (I-P) ( \nabla_X (Z+iJZ) - \langle X,Z + iJZ \rangle n_x ) \\
\hphantom{C(Z+iJZ)}{} =(I-P) ( \nabla_X (Z+iJZ) ) = P^+ ( \nabla_X (Z+iJZ))
\end{gather*}
and
\begin{gather*}
 P^+ ( \nabla_X (Z+iJZ) ) = P^+ ( \nabla_X (P^- Z) ) \\
 \hphantom{P^+ ( \nabla_X (Z+iJZ) )}{} = \nabla_X \big( P^+ P^- Z\big)- \nabla_X\big(P^+\big) (P^-Z) = \frac{i}{2} (\nabla_X J) (Z + i J Z).
 \end{gather*}
 Finally,
\begin{gather*}
C(n_x) = \frac 1 2 (X - i J X)
\end{gather*}
as in the proof of Lemma~\ref{lem;Pcanonemb}.

It follows that if $d_YP =\left[
\begin{smallmatrix}
0 & D \\
E& 0
\end{smallmatrix}\right]$ then
\begin{gather*}
D(Z-iJZ) = \frac i 2 (\nabla_Y J) (Z-iJZ) + \langle Y , Z-iJZ \rangle n_x, \\
E (Z+iJZ) = \frac{i}{ 2} (\nabla_Y J)(Z+iJZ) ,\\
E(n_x) = \frac 1 2 (Y - i J Y),
\end{gather*}
so that
\begin{gather*}
BE (Z+iJZ) = -\frac{1}{4} (\nabla_X J) (\nabla_Y J)(Z+iJZ) + \left\langle X , \frac{i}{ 2} (\nabla_Y J)(Z+iJZ) \right\rangle n_x, \\
BE(n_x) = \frac i 4 (\nabla_X J) (Y - i J Y) + \frac 1 2 \langle X , (Y - i J Y) \rangle n_x, \\
DC(Z+iJZ) = -\frac{1}{4} (\nabla_Y J) (\nabla_X J)(Z+iJZ) + \left\langle Y , \frac{i}{ 2} ( \nabla_X J (Z+iJZ)) \right\rangle n_x \\
DC(n_x) = \frac i 4 (\nabla_Y J) (X-iJX) + \frac 1 2\langle Y , X-iJX \rangle n_x.
\end{gather*}

The matrix $BE-DC$ is a $\CC$-linear endomorphism of $\operatorname{Im}(P) = T^{0,1}_x \oplus (N_x \otimes \CC)$. Let $Z_k + i J Z_k$ for $k=1,\ldots ,n$ be an orthonormal basis over $\CC$ for $T^{0,1}_x$, so that this set along with the vector~$n_x$ is an orthonormal basis over $\CC$ for $\operatorname{Im}(P)$. For the remainder of this proof, let $\langle\langle - , - \rangle\rangle $ denote the standard inner product on $\CC^{2n+1}$, which is $\CC$-linear in the first entry, and conjugate linear in the second entry.

Then
\begin{gather*}
\Tr(BE-DC) = \sum_k \langle\langle (BE-DC) (Z_k + i J Z_k), Z_k + i J Z_k \rangle\rangle + \langle\langle (BE-DC)n_x,n_x \rangle\rangle \\
=\sum_k \langle\langle (BE-DC)( Z_k + i J Z_k), Z_k + i J Z_k \rangle\rangle + \frac 1 2 \langle X , (Y - i J Y) \rangle - \frac 1 2 \langle Y , X-iJX \rangle \\
= \sum_k \left\langle\left\langle -\frac{1}{4} (\nabla_X J) (\nabla_Y J)(Z_k + i J Z_k) +
\frac{1}{4} (\nabla_Y J) (\nabla_X J)(Z_k + i J Z_k)
, Z_k + i J Z_k \right\rangle\right\rangle \\
\quad{} + \frac 1 2 \langle X , Y\rangle - \frac 1 2 \langle Y , X \rangle + \frac i 2 ( \langle Y , JX \rangle - \langle X , JY \rangle ) \\
= \frac{1}{4} \sum_k \langle\langle [ \nabla_Y J , \nabla_X J ] (Z_k + i J Z_k)
, Z_k + i J Z_k \rangle\rangle + \frac i 2 ( \langle Y , JX \rangle - \langle X , JY \rangle ) .
\end{gather*}
Since $\operatorname{Re}(\omega) = \operatorname{Re} \big( \frac{1}{2i} \Tr(BE-DC)\big)$ we have
\begin{gather*}
 P^* (\operatorname{Re} \omega ) (X,Y)
= -\frac{1}{8} \operatorname{Im} \Tr \big( [ \nabla_X J , \nabla_Y J ] \big|_{T^{0,1}} \big)
 - \frac 1 4 \langle X,JY \rangle + \frac 1 4 \langle Y,JX \rangle. \tag*{\qed}
 \end{gather*}\renewcommand{\qed}{}
\end{proof}

For the following corollaries we continue to consider the round sphere with standard metric $\langle - ,- \rangle$ and Levi-Civita connection~$\nabla$.

\begin{cor} \label{cor;commute} No almost complex structure $J$ on $S^6$ satisfies\footnote{This result, and all of those below that rely on $H^2\big(S^6\big) = 0$, can be sharpened to state that the ``disqualifying condition'' cannot happen along the image of any closed $J$-holomorphic curve $\Sigma \to S^6$ which is an immersion on a~full measure subset, by Stokes' theorem. One can show that for an integrable $J$ on $S^6$, there is at most a~discrete set of (image) curves which are not multiply covered. The sharpened statements will not be mentioned in the subsequent corollaries.}
that $\nabla_X J$ and $\nabla_{JX} J$ are eve\-ry\-where commuting for all vectors~$X$.
\end{cor}

\begin{proof}By Theorem~\ref{thm;Pcanon}, if $\nabla_X J$ and $\nabla_{JX} J$ always commute then
\begin{gather*}
P^* (\operatorname{Re} \omega ) (X,JX) = \frac 1 4 \langle X,X\rangle + \frac 1 4 \langle JX,JX\rangle > 0,
\end{gather*}
implying $S^6$ has a closed non-degenerate $2$-form, which is a contradiction since $H^2\big(S^6\big) = 0$.
\end{proof}

More generally, we conclude that the topology of $S^6$ restricts the possible imaginary trace of $ [ \nabla_X J , \nabla_{JX} J ] \big|_{T^{0,1}} $.

\begin{cor}For any almost complex structure $J$ on $S^6$, there is some non-zero tangent vector $X$ at some point in $S^6$ such that
\begin{gather*}
\operatorname{Im} \Tr \big( [ \nabla_X J , \nabla_{JX} J] \big|_{T^{0,1}} \big) \geq 2 \big( \|X \|^2 + \|JX \|^2 \big).
\end{gather*}
In particular, there is a unit tangent vector $X$ for which
\begin{gather*}
\operatorname{Im} \Tr \big( [ \nabla_X J , \nabla_{JX} J ] \big|_{T^{0,1}} \big) \geq 2.
\end{gather*}
\end{cor}

\begin{proof}
If the inequality were false for all $X$, then by Theorem \ref{thm;Pcanon} $P^*(\operatorname{Re} \omega)$ would be non-degenerate, thereby contradicting $H^2\big(S^6\big) = 0$. The second case follows from the first with $\|X\|=1$.
\end{proof}

\begin{cor}[Blanchard \cite{Bl}, Lebrun \cite{L}] \label{cor;L}No orthogonal almost complex structure on the round sphere $S^6$ is integrable.
\end{cor}

Both \cite{Bl} and \cite{L} use in their arguments a certain $J$-holomorphic map, which is closely related to the canonical map we're using here. A~similar approach will be taken in the next section to conclude a stronger result. Here we'll re-prove the claim by calculation, using Theorem~\ref{thm;Pcanon}.

\begin{proof}
If $J$ were integrable and orthogonal, then by Theorem \ref{thm;Jintorth}, $J$ is strongly integrable with respect to $\nabla$, i.e., $\nabla_{JX} J = J \nabla_X J$
for any $X$. Then
\begin{gather*}
 [\nabla_X J, \nabla_{JX} J] = ( \nabla_X J J \nabla_{X} J - J \nabla_X J \nabla_{X} J )=-2 J ( \nabla_X J )^2,
\end{gather*}
so that
\begin{gather*}
\big( P^*\operatorname{Re} \omega\big)(X,JX) = \frac{1}{4} \operatorname{Im} \Tr\big( J ( \nabla_X J )^2 \big|_{T^{0,1}} \big) + \frac 1 4 \langle X,X\rangle + \frac 1 4 \langle JX,JX \rangle .
\end{gather*}
Now, letting $Z_k +iJZ_k$ for $k= 1, \ldots , n$ be an orthonormal basis for $T^{0,1}$ we have
\begin{gather*}
\operatorname{Im} \Tr\big( J ( \nabla_X J )^2 \big|_{T^{0,1}} \big) = \operatorname{Im} \sum_k \big\langle\big\langle J ( \nabla_X J )^2 (Z_k +iJZ_k), Z_k +iJZ_k \big\rangle\big\rangle \\
\hphantom{\operatorname{Im} \Tr\big( J ( \nabla_X J )^2 \big|_{T^{0,1}} \big)}{} =
\sum_k \big( \big\langle\big\langle J (\nabla_X J)^2 JZ_k, Z_k \big\rangle\big\rangle -
 \big\langle\big\langle J (\nabla_X J)^2 Z_k , JZ_k \big\rangle\big\rangle \big) \\
\hphantom{\operatorname{Im} \Tr\big( J ( \nabla_X J )^2 \big|_{T^{0,1}} \big)}{}
 = -2 \sum_k \big\langle\big\langle (\nabla_X J)^2 Z_k, Z_k \big\rangle\big\rangle = 2 \sum_k \| (\nabla_X J) Z_k \|^2,
\end{gather*}
where in the last step we used that $J$ is skew, so that $(\nabla_X J)^* = ( \nabla_X J^* ) = - \nabla_X J $. This shows $(P^* \operatorname{Re} \omega)(X,JX) > 0$, which again contradicts $H^2\big(S^6\big) = 0$.
\end{proof}

It does not appear that using Theorem \ref{thm;Pcanon} alone we can relax the previous theorem's hypotheses (of integrable and orthogonal)
to strong integrability; it is possible that $\operatorname{Im} \Tr\big( J ( \nabla_X J )^2 \big|_{T^{0,1}} \big)$ is negative and large in absolute value. In the next section we provide a more conceptual way to show that there are no almost complex structures on $S^6$ which are strongly integrable with respect to the Levi-Civita connection on the round sphere.

\section{Grassmann map} \label{sec;Gmap}

In this section we let $n=1$ or $n=3$, $S^{2n} \subset \R^{2n+1}$ be the unit sphere with respect to the standard Euclidean metic, and let $\nabla$ be the Levi-Civita connection. Recall that at any point $x \in S^{2n}$ we have
\begin{gather*}
\CC^{2n+1} = \big( T^{0,1}_x \oplus (N_x \otimes \CC) \big) \oplus T^{1,0}_x.
\end{gather*}

\begin{defn}For any $J$ on $S^{2n}$, the induced map $P^\perp\colon S^{2n} \to \cG_{n+1,2n+1}$ is defined so that~$P^\perp(x)$ is the orthogonal projection of $\CC^{2n+1}$ onto $T^{0,1}_x \oplus (N_x \otimes \CC)$.
\end{defn}

Lebrun \cite{L} considers the closely related map which is given by orthogonal projection onto~$T^{0,1}_x$, and shows it is an embedding. We'll see below that by including the normal bundle in the new map above, the prohibited condition on~$S^6$ becomes ``strongly integrable'' rather than ``integrable and orthogonal''. This again is motivated by the normal component in equation~\eqref{eq;NablaSphere}.

Theorem \ref{thm;delbarPperp} and Corollary~\ref{cor;noStrongJ} below are very much inspired by the author's reading of~\cite{L}. First, we show this map is an immersion, which is sufficient for our purposes.

\begin{lem} For any $J$ on $S^{2n}$, the map $P^\perp\colon S^{2n} \to \cG_{n+1,2n+1}$ is an immersion.
\end{lem}

\begin{proof}If $X \in T_x S^{2n}$ then $d_X P \in T_{P(x)} \cG_{n+1,2n+1}$ is a non-zero matrix since, for the real normal unit vector $n_x$ to $S^{2n}$ at $x$, we have
\begin{gather*}
\big(d_XP^\perp\big)(n_x) = d_X\big(P^\perp(n_x)\big) - P^\perp d_X(n_x) = \big(I-P^\perp \big) d_X(n_x) = \big(I-P^\perp\big)(X).
\end{gather*}
This shows $d_X P$ is non-zero whenever $X$ is non-zero since $X \in T_xM$ is real, but $P^\perp(X) \in T^{0,1}$ is not real.
\end{proof}

Recall that the tangent algebra $m$ associated to $J$ and $\nabla$ is given by
\begin{gather*}
m(X,Y) = \nabla_{JX} (J)Y - J \nabla_X (J) Y,
\end{gather*}
and that $J $ is strongly integrable with respect to $\nabla$ if and only if $\nabla_X J$ is $J$-linear in the variable~$X$, i.e., $m$ vanishes.

For a map $f\colon (M,J) \to (N,K) $ between almost complex manifolds, we define
\begin{gather*}
\overline{\partial}_X f = d_{JX} f - K \circ d_X f.
\end{gather*}

\begin{thm} \label{thm;delbarPperp}
For any $J$ on $S^{2n}$, the map $P^\perp\colon S^{2n} \to \cG_{n+1,2n+1}$ induced by $J$ satisfies
\begin{gather*}
\overline{\partial}_X P^\perp (n_x) = 0, \\
\overline{\partial}_X P^\perp (Y+iJY) = \big( I - P^\perp\big) (J m(X,Y)),
\end{gather*}
where $m(X,Y) = \nabla_{JX} (J)Y - J \nabla_X (J) Y$ is the tangent algebra.

Furthermore, $\overline{\partial} P^\perp \equiv 0$ if and only if~$J $ is strongly integrable with respect to~$\nabla$, i.e., the tangent algebra $m$ vanishes.
\end{thm}

\begin{proof}Recall from Section \ref{sec;grass} that the complex structure on $T_P \cG_{n+1,2n+1}$ is given by multiplication by $+i$ on $\operatorname{Hom}\big(\operatorname{Im}\big(P^\perp\big), \operatorname{Ker}\big(P^\perp\big)\big)$, and by self-adjointness, is given by multiplication by $-i$ on $\operatorname{Hom}\big(\operatorname{Ker}\big(P^\perp\big),\operatorname{Im}\big(P^\perp\big)\big)$.

As in the previous lemma we have
\begin{gather*}
 \overline{\partial}_X P^\perp (n_x) = -i \big( d_{X+iJX} P^\perp \big) (n_x) = -i \big(I-P^\perp\big)(X+iJX) = 0,
\end{gather*}
and, for arbitrary $Z \in T^{0,1}_x$, we have
\begin{gather*}
\big( d_{X+iJX} P^\perp \big) (Z) = \big(I-P^\perp\big) d_{X+iJX} (Z)= \big(I-P^\perp\big) ( \nabla_{X+iJX} (Z) - \langle X+iJX, Z \rangle n_x ) \\
\hphantom{\big( d_{X+iJX} P^\perp \big) (Z)}{} = \big(I-P^\perp\big) ( \nabla_{X+iJX} (Z) ),
 \end{gather*}
 since $P^\perp(n_x) = n_x$. Now, if $Z= Y+iJY$, then the last expression is equal to
 \begin{gather*}
 \big(I -P^\perp\big) \big( \nabla_X Y - \nabla_{JX} (JY) + i ( \nabla_X(JY) + \nabla_{JX} Y ) \big) \\
 \qquad{} = \big(I -P^\perp\big) i \big( \nabla_X(JY) + \nabla_{JX} Y - J(\nabla_X Y - \nabla_{JX} (JY)) \big) \\
 \qquad {}= i \big(I -P^\perp\big) \big( J ( \nabla_{JX} (J) Y - J\nabla_X(J) Y ) \big) = i \big(I - P^\perp\big) ( J m(X,Y) ).
\end{gather*}
 So,
\begin{gather*}
\overline{\partial}_X P^\perp (Y+iJY) = -i \big( d_{X+iJX} P^\perp \big) (Y + iJY) = ( I - P^\perp) ( J m(X,Y) ),
\end{gather*}
as claimed. Since $ J m(X,Y)$ is real, the right hand side vanishes if and only if~$m$ vanishes, i.e., $J $~is strongly integrable with respect to $\nabla$.

The remainder of the proof, showing that $\overline{\partial} P^\perp \equiv 0$ when $m$ vanishes, is entirely formal since the restriction $dP^\perp\colon \big(T^{0,1}\big)^\perp \to T^{0,1}$ is the adjoint of the restriction $dP^\perp\colon T^{0,1} \to \big(T^{0,1}\big)^\perp$. Explicitly, for any $W + i JZ = \big(T^{0,1}\big)^\perp$ we have with respect to the complex inner product $\langle - , - \rangle$,
\begin{gather*}
\big\langle \big( d_{X-iJX} P^\perp \big) (W+iJZ) , Y+i JY \big\rangle
= \big\langle W+iJZ , \big( d_{X+iJX} P^\perp \big)( Y+i JY )\big\rangle = 0
\end{gather*}
since $P^\perp$ is always self adjoint. Also, for $U + iV \in \big(T^{0,1}\big)^\perp$, we have
\begin{gather*}
\big\langle \big( d_{X-iJX} P^\perp \big) (W+iJZ) , U+iJV \big\rangle = 0,
\end{gather*}
since $d_{X-iJX} P^\perp \colon \big(T^{0,1}\big)^\perp \to T^{0,1}$, because $P^\perp$ is a projection operator.
\end{proof}

Note that by Theorem~\ref{thm;Jintorth}, any integrable orthogonal $J$ on~$S^6$ is strongly integrable with respect to $\nabla$, whereas locally there are non-orthogonal almost complex structures on~$S^6$ which are strongly integrable with respect to a torsion free connection. For example, from Euclidean space we have those non-orthogonal almost complex structures which are covariantly constant. Nevertheless, we can conclude:

\begin{cor} \label{cor;noStrongJ}There are no almost complex structures on $S^6$ which are strongly integrable with respect to the Levi-Civita connection~$\nabla$ of the standard metric.
\end{cor}

\begin{proof} By the Theorem~\ref{thm;delbarPperp}, if $J$ is strongly integrable with respect to $\nabla$, then $P^\perp\colon S^6 \to \cG_{4,7}$ is a $J$-holomorphic immersion. Since $\cG_{4,7}$ is a K\"ahler manifold, this shows the pullback $\big(P^\perp \big)^*\omega$ is closed and non-degenerate, which contradicts $H^2\big(S^6\big) =0$.
\end{proof}

Note that in Theorem \ref{thm;delbarPperp} and Corollary~\ref{cor;noStrongJ} we did not use the explicit formula for the K\"ahler form $\omega$, but rather only the existence of such a structure compatible with the complex structure on~$\cG_{k,n}$. While this was sufficient to conclude the pullback $\big(P^\perp\big)^*\omega$ is non-degenerate, it is far from necessary. We next give a necessary and sufficient condition for $\big(P^\perp\big)^*\omega$ to be non-degenerate, which is intrinsic in that it depends only on~$J$, the metric and the covariant derivatives of~$J$.

To do so, we calculate the pullback form $\big(P^\perp\big)^*\omega$ explicitly, but first we need a lemma. This lemma should be thought of as the non-orthogonal analogue of the standard way in which any orthogonal $J$ determines a compatible $2$-form tamed by~$J$.

\begin{lem} \label{lem;JRM} Let $V$ be a real vector space with complex structure $J$ and real inner product $\langle - ,- \rangle$. Extend $\langle - ,- \rangle$ to a complex inner product on $W = V \otimes \CC \cong V \oplus i V$ and let $Q\colon W \to W$ be orthogonal projection onto $T^{0,1}(V,J)$, i.e., the $-i$ eigenspace of $J$. Then $Q = R + i M$ for unique real linear operators~$R$ and~$M$ on~$V$, and we have
\begin{itemize}\itemsep=0pt
\item[$(1)$] $R-MJ= \operatorname{Id}$;
\item[$(2)$] $R^* = R$ and $M^* = -M$;
\item[$(3)$] $\Tr(Q) = \dim(V)/2$;
\item[$(4)$] $R^2-M^2 = R$ and $RM+MR= M$.
\end{itemize}
Conversely, any real operators $R$ and $M$ on~$V$ satisfying the above properties determine the operator $Q=R+iM$ to be the orthogonal projection onto $T^{0,1}$, and satisfy
\begin{itemize}\itemsep=0pt
\item[$(a)$] $MJ= - J^*M$;
\item[$(b)$] $ M^2J-J^*M^2 = -M$;
\item[$(c)$] $\nu(X,Y) := \langle M X,Y\rangle$ is a $2$-form, tamed by and compatible with $J$, i.e., there is an associated positive definite inner product
\begin{gather*}
(X,Y) := \nu(X,JY),
\end{gather*}
and $J$ is orthogonal with respect to $\nu$ and $( \, , \, )$.
\end{itemize}
\end{lem}

\begin{proof} Let $R = \operatorname{Re}(Q)$ and $M = \operatorname{Im}(Q)$. The first condition follows since~$Q$ is the identity on~$T^{0,1}$
\begin{gather*}
(R+iM)(Z+iJZ) = (R-MJ)Z + i(RJ+MZ) = Z+iJZ.
\end{gather*}
The second condition follows since $Q$ is self adjoint, the third condition since $Q$ is a projection of rank $\dim(V)/2$, and the fourth since $Q^2 = Q$. The converse holds since orthogonal projection onto $T^{0,1}$ is the unique self adjoint projection of rank $\dim(V)/2$ and image $T^{0,1}$. Condition~(a) follows from~$(1)$ and~$(2)$ since $R-\operatorname{Id}$ is self adjoint, and $M$ is skew-adjoint. Next, condition $(b)$ is given by
\begin{gather*}
\big(M^2J-J^*M^2) = MMJ+MJM = M(R-I)+(R-I)M = MR+RM-2M = -M.
\end{gather*}
The inequality
\begin{gather*}
0 \leq \|X\|^2 - \|Q X\|^2 = \langle (I-Q)X, X \rangle = \langle (I-R)X, X \rangle + i \langle MX , X\rangle
\end{gather*}
shows $\langle MX , X\rangle =0$ and
\begin{gather*}
 (X,X) = \langle -MJ X, X \rangle = \langle (I-R)X, X \rangle \geq 0.
\end{gather*}
Finally, $\eta$ is alternating by $(2)$, and the rest follows
\begin{gather*}
(X,Y) = \langle M X, JY\rangle = \langle X, J^*M Y\rangle = \langle JX, MY\rangle = (Y,X),
\end{gather*}
and
\begin{gather*}
 \nu(JX,JY) = \langle M JX, JY\rangle = \langle - J^* M X, JY\rangle = \nu(X,Y) ,
\end{gather*}
and
\begin{gather*}
(JX,JY) = -\nu(JX,Y)= \nu(Y,JX)= (Y,X) = (X,Y).\tag*{\qed}
 \end{gather*}
 \renewcommand{\qed}{}
\end{proof}

\begin{ex} \label{ex;MonehalfJ} If $J$ is orthogonal on $V$ with respect to a given inner product, then $R= \frac 1 2 \operatorname{Id} $ and $M = \frac 1 2 J$. The last condition in the lemma is then a multiple of the usual assignment $\omega(X,Y) := (J X,Y)$.
 \end{ex}

We can apply the preceding algebraic lemma at every tangent space to give an explicit description of the pullback form $\big(P^\perp\big)^*\omega$ along the map $P^\perp\colon S^{2n} \to \cG_{n+1,2n+1}$.

\begin{prop} \label{prop;pull}The pullback of the K\"ahler form $\omega$ on $ \cG_{n+1,2n+1}$ along $P^\perp$ is given by
\begin{gather} \label{eq;pullomega}
\big(P^\perp\big)^*\omega(X,Y) = \sum_k \nu \big( (\nabla_X J) Z_k, (\nabla_Y J) Z_k \big) + \nu (X , Y),
\end{gather}
where the sum is over any $Z_k$ such that $\{Z_k + i J Z_k\}$ is an orthonormal basis of $T^{0,1}$, and $\nu(\alpha ,\beta) = \langle M \alpha, \beta \rangle$ is the associated $2$-form from Lemma~{\rm \ref{lem;JRM}}.
\end{prop}

\begin{proof}We first show that
\begin{gather*}
d_X P^\perp\big|_{T^{0,1}} = M \nabla_X J.
\end{gather*}
Letting $Z + iJZ \in T^{0,1}$ be arbitrary, we calculate
\begin{gather*}
d_X P^\perp (Z + i JZ) = \big(I-P^\perp\big) d_X (Z + iJZ)= \big(I-P^\perp\big) \big( \nabla_X(Z+iJZ) -\langle X,Z + i JZ\rangle n_x \big) \\
\hphantom{d_X P^\perp (Z + i JZ)}{} = \big(I-P^\perp\big) (\nabla_X(Z+iJZ)).
\end{gather*}
By definition $P^\perp = Q \stackrel{{\perp}}{\oplus} \pi_{N}$ where $Q$ is the orthogonal projection of~$TS^{2n} \otimes \CC$ onto $T^{0,1}$ and~$\pi_N$ is orthogonal projection onto the normal bundle, so
\begin{gather*}
d_X P^\perp (Z + i JZ) =(I- Q \oplus \pi_{N}) ( \nabla_X(Z+iJZ) ) \\
\hphantom{d_X P^\perp (Z + i JZ)}{} = (I-Q) ( \nabla_X(Z+iJZ) ) = (\nabla_X Q) (Z+iJZ).
\end{gather*}
Now, letting $Q= R + iM$, as in the previous lemma,
\begin{gather*}
(\nabla_X R + i \nabla_X M) (Z+iJZ) = ( \nabla_X R ) Z - (\nabla_X M) (JZ)
+ i \big( (\nabla_X M) (Z) + (\nabla_X R) (JZ) \big).
\end{gather*}
By the lemma, $R-MJ= \operatorname{Id}$ so that
\begin{gather*}
\nabla_X R - (\nabla_X M) J - M (\nabla_X J)= 0,
\end{gather*}
so we conclude
\begin{gather*}
d_X P^\perp (Z + i JZ) = M \nabla_X (J) Z + i M \nabla_X(J) (J Z) = M \nabla_X(J) (Z+iJZ).
\end{gather*}

Recall that the K\"ahler form $\omega(X,Y)$ on $\cG_{k,n}$ is given by $\operatorname{Im} \Tr(BD^*)$, so that the pull\-back~$\big(P^\perp\big)^*\omega$ is given by
\begin{gather*}
\big(P^\perp\big)^*\omega(X,Y)= \operatorname{Im} \Tr \big( d_X P^\perp\big|_{(T^{0,1} \oplus N)^\perp} \circ d_Y P^\perp \big|_{T^{0,1}\oplus N} \big),
\end{gather*}
where we have made the substitutions $B= d_X P^\perp\big|_{(T^{0,1}\oplus N)^\perp}\colon \operatorname{Ker}\big(P^\perp\big) \to \operatorname{Im}\big(P^\perp\big)$ and $D^* = d_Y P^\perp\big|_{T^{0,1} \oplus N} \colon \operatorname{Im}\big(P^\perp\big) \to \operatorname{Ker}\big(P^\perp\big)$.

For any orthonormal basis $\{Z_k + i J Z_k\} \cup \{n\}$ of $T^{0,1} \oplus N$ with respect to the complex inner product $\langle - ,- \rangle$, which is conjugate linear in the second variable, we use the fact that the adjoint of $B$ is $B^* = d_X P^\perp\big|_{T^{0,1} \oplus N} \colon \operatorname{Im}\big(P^\perp\big) \to \operatorname{Ker}\big(P^\perp\big)$, and have
\begin{gather*}
\big(P^\perp\big)^*\omega(X,Y) = \operatorname{Im} \big\langle \big( d_Y P^\perp\big|_{T^{0,1}\oplus N} \big) n, \big( d_X P^\perp\big|_{T^{0,1} \oplus N} \big) n \big\rangle\\
\hphantom{\big(P^\perp\big)^*\omega(X,Y) =}{} + \sum _k \operatorname{Im} \big\langle \big( d_Y P^\perp\big|_{T^{0,1} \oplus N} \big) (Z_k + i J Z_k), \big( d_X P^\perp\big|_{T^{0,1}\oplus N} \big) (Z_k + i J Z_k) \big\rangle.
\end{gather*}
For the first summand on the right hand side we use
\begin{gather*}
\big( d_X P^\perp\big|_{T^{0,1} \oplus N } \big) n = \big(I-P^\perp\big) \nabla_X n = \big(I-P^\perp\big)(X) = X - RX - iMX,
\end{gather*}
and Lemma \ref{lem;JRM} to obtain
\begin{gather*}
\operatorname{Im} \big\langle \big( d_Y P^\perp\big|_{T^{0,1}\oplus N} \big) n, \big( d_X P^\perp\big|_{T^{0,1}\oplus N} \big) n \big\rangle
 = \langle Y - RY - iMY, X - RX - iMX \rangle \\
 \qquad{} = \langle (I-R)Y ,MX \rangle - \langle MY, (I-R)X\rangle = \langle MX, Y \rangle.
\end{gather*}

For the latter summation term we have
\begin{gather*}
\operatorname{Im} \big\langle \big( d_Y P^\perp\big|_{T^{0,1}} \big) (Z_k + i J Z_k), \big( d_X P^\perp\big|_{T^{0,1}} \big) (Z_k + i J Z_k) \big\rangle\\
\qquad{} =\langle ( M \nabla_Y J ) (J Z_k), ( M \nabla_X J ) (Z_k ) \rangle
-\langle ( M \nabla_Y J ) (Z_k ), ( M \nabla_X J ) ( J Z_k) \rangle \\
\qquad{} =\big\langle \big( M^2J \nabla_Y J\big) ( Z_k), (\nabla_X J) (Z_k ) \big\rangle-
\langle ( M \nabla_Y J ) (Z_k ), -MJ (\nabla_X J) ( Z_k) \rangle \\
\qquad{} = \big\langle \big(M^2J-J^*M^2\big) ( \nabla_Y J ) ( Z_k), (\nabla_X J) (Z_k ) \big\rangle
=\langle M (\nabla_X J) ( Z_k), (\nabla_Y J) (Z_k ) \rangle ,
\end{gather*}
where in the last step we have we have used Lemma~\ref{lem;JRM}.
\end{proof}

\begin{rmk}Proposition \ref{prop;pull} also verifies (in another way) Corollary~\ref{cor;noStrongJ}, for if $\nabla_{JX} J = J \nabla_X J$, then by Lemma~\ref{lem;JRM} part $(c)$ we have
\begin{gather*}
\big(P^\perp\big)^*\omega(X,JX) = \sum_k \| (\nabla_X J) Z_k \|^2 + \|X\|^2 >0,
\end{gather*}
where $\|\cdot \|$is the norm induced by the inner product $( \alpha ,\beta ) := \nu (\alpha , J \beta)$. In the case that $J$ is orthogonal, this agrees with the proof in Corollary~\ref{cor;L}, by Example~\ref{ex;MonehalfJ}.
\end{rmk}

For the remainder of this section, consider the $2$-form given by
\begin{gather*}
\eta(X,Y) = \sum_k \nu \big( (\nabla_X J) Z_k, (\nabla_Y J) Z_k \big),
\end{gather*}
where the sum is over any $Z_k$ such that $\{Z_k + i J Z_k\}$ is an orthonormal basis of $T^{0,1}$, and $\nu(\alpha ,\beta) = \langle M \alpha, \beta \rangle$ is the associated $2$-form from Lemma~\ref{lem;JRM}. It is an artifact of the previous proof that this form $\eta$ on $S^6$ does not depend on the choice of $\{Z_k\}$, and we have
\begin{gather*}
\big(P^\perp\big)^*\omega(X,Y) =\eta(X,Y) + \nu ( X , Y ).
\end{gather*}
The necessary degeneracy of this form on $S^6$ implies

\begin{cor} For any $J$ on $S^{6}$, integrable or not, there are some tangent vectors $X$, $Y$ at some point for which
\begin{gather*}
 \eta(X,Y) + \nu( X , Y ) = 0.
\end{gather*}
\end{cor}

Let $( - , - )$ be the induced positive definite inner product from Lemma~\ref{lem;JRM} defined by $ ( X,Y) = \nu(X,JY) = \langle MX, JY \rangle$, with associated norm $\| \cdot \|$.

\begin{cor} \label{cor;bound}If $J$ on $S^6$ satisfies the first order differential inequality
\begin{gather} \label{eq;bound}
\sum_k \big( (\nabla_X J) Z_k, J m(Z_k,X) \big) \neq \sum_k \| (\nabla_X J) Z_k \|^2 + \|X\|^2
\end{gather}
for all non-zero $X$, then $J$ is not integrable.

In particular, if the tangent algebra $m$ satisfies
\begin{gather} \label{eq;halfbound}
\big( ( \nabla_X J ) Z, J m(Z,X) \big) \leq \| ( \nabla_X J ) Z \|^2 ,
\end{gather}
for all $X$ and $Z$, or the weaker condition
\begin{gather} \label{eq;diskbound}
\| m(Z,X) \| \leq \| (\nabla_X J) Z \|,
\end{gather}
for all $X$ and $Z$, then $J$ is not integrable.
\end{cor}

There are a plethora of local almost complex structures in dimension $6$ for which equation~\eqref{eq;bound} does hold, since it is an open condition that holds for every local integrable orthogonal almost complex structure. Note that geometrically, condition equation~\eqref{eq;halfbound} is an open half-space restriction on~$m$ in terms of~$\nabla J$, whereas equation~\eqref{eq;diskbound} is an open sub-disk restriction.

\begin{proof}Recall from equation~\eqref{eq:Nnabla}, $J$ is integrable if and only if
\begin{gather*}
 (\nabla_Y J) Z = -J (\nabla_{JY} J ) Z + J (\nabla_{JZ} J - J \nabla_Z J)Y = -J ( \nabla_{JY} J) Z + Jm(Z,Y)
\end{gather*}
for all vectors $Y$ and $Z$. This condition gives
\begin{gather*}
\eta(X,Y)=\sum_k \nu \big((\nabla_X J) Z_k, -J ( \nabla_{JY} J ) Z_k\big)+\sum_k \nu \big( (\nabla_X J) Z_k, J m(Z_k,Y) \big).
\end{gather*}
Now letting $Y=JX$ we have
\begin{gather*}
\eta(X,JX) = \sum_k \| (\nabla_X J) Z_k \|^2- \sum_k \big( (\nabla_X J) Z_k, Jm(Z_k,X) \big),
\end{gather*}
where we have used $m(Z_k,JX) = -Jm(Z_k,X)$. Since
\begin{gather*}
\big(P^\perp\big)^*\omega(X,JX) =\eta(X,JX) + \nu( X , JX ) = \eta(X,JX) + \|X\|^2,
\end{gather*}
we conclude that if
\begin{gather*}
 \sum_k \| (\nabla_X J) Z_k \|^2-\sum_k \big( (\nabla_X J) Z_k, Jm(Z_k,X) \big) + \|X\|^2 \neq 0
\end{gather*}
for all $X$, then $J$ is not integrable. The final claim follows from the Cauchy--Schwarz inequa\-lity.
\end{proof}

The remainder of this section is meant to explain the meaning and applicability of the previous corollary. First, we have already seen an example that distinguishes between the conditions in Corollary~\ref{cor;bound}.

\begin{ex} The above Example \ref{ex;2dex} in dimension $2$ provides a $J$ which is not strongly integrable with respect to the standard metric, the bound in equation~\eqref{eq;diskbound} does not hold, but the bound in equation~\eqref{eq;halfbound} holds for $0 < x < 1$, and therefore equation~\eqref{eq;bound} holds for such $x$ as well. Therefore, the bound in equation~\eqref{eq;diskbound} is a priori weaker than the bound in equation~\eqref{eq;bound}, which is a priori weaker than being strongly integrable.
\end{ex}

 To gain some further perspective on the meaning of equation~\eqref{eq;bound} in Corollary \ref{cor;bound}, it is instructive to consider the situation for a general $J$ on $S^2$, with the standard metric, for which we have the following surprisingly pleasant geometric result.

\begin{prop} \label{prop;S2calc} Let $J$ be an arbitrary almost complex structure on $S^2$. The pullback form $\big(P^\perp\big)^*\omega \in \Om^2(S^2)$ is non-degenerate at $p \in S^2$ if and only if for some vector $X_p$ at $p$,
\begin{gather*}
\det(dF_p) \neq \|X_p + i JX_p\|^2,
\end{gather*}
where $dF_p$ is the Jacobian of the smooth mapping $F_p: T_pS^2 \to T_p S^2$ defined by $F_p(w) = \tilde J_w(X_p)$, and the almost complex structure $\tilde J$ on the manifold $T_pS^2$ is given by the pullback of $J$ on $S^2$ by stereographic projection.
\end{prop}

Since every almost complex structure in dimension two is integrable, and for integrable almost complex structures~\eqref{eq;bound} holds precisely when $\big(P^\perp\big)^*\omega$ is non-degenerate, this determines all almost complex structures on~$S^2$ for which~\eqref{eq;bound} holds.

\begin{proof}We perform a local calculation of the form $\big(P^\perp\big)^*\omega$ at a given point $p$ in the sphere using stereographic projection. Any $J$ on $S^2$ can be pulled back via these coordinates to $T_pS^2$, which can be written uniquely in the standard basis $\{\partial_x , \partial_y\}$ as
\begin{gather*}
\tilde J=
\begin{bmatrix}
f & \frac{-1-f^2}{g}\\
g & -f
\end{bmatrix}
\end{gather*}
for some functions $f,g\colon \R^2 \to \R$, with $g$ nowhere zero. Moreover, on any fixed disk about the origin in $T_pS^2$, every such $\tilde J$ occurs in this way.

Without loss of generality, we assume $p\in S^2$ is the south pole of the sphere, with stereographic projection defined by projection on the tangent plane at the south pole. The standard metric on the sphere in these stereographic coordinates is given by
\begin{gather*}
ds^2 = h(x,y) \big(dx^2 + dy^2\big), \qquad \textrm{where} \qquad h(x,y) = \frac{1}{\big(1+x^2 +y^2\big)^2}.
\end{gather*}

Using the expression for the Christoffel symbols in terms of the metric,
\begin{gather*}
\Gamma_{jk}^\ell = \frac 1 2 \big(ds^2\big)^{lr} \big( \partial_k \big(ds^2\big)_{rj} + \partial_j \big(ds^2\big)_{rk} - \partial_r \big(ds^2\big)_{jk} \big),
\end{gather*}
we have
\begin{gather*}
\Gamma_{xx}^x = \Gamma_{xy}^y = \Gamma_{yx}^y = - \Gamma_{yy}^x = \frac{h_x}{2h}, \qquad
\Gamma_{yy}^y = \Gamma_{xy}^x = \Gamma_{yx}^x = - \Gamma_{xx}^y = \frac{h_y}{2h}.
\end{gather*}
One then calculates
\begin{gather}
( \nabla_{\partial_x} J ) \partial_x= \left(f_x + \frac{h_y}{2h} \left( g- \frac{1+f^2}{g} \right) \right) \partial_x +
\left(g_x - \frac{f h_y}{h}\right) \partial_y, \label{eqn;NJ1} \\
( \nabla_{\partial_y} J ) \partial_x = \left(f_y - \frac{h_x}{2h} \left( g - \frac{ \big(1+f^2 \big) }{g} \right) \right) \partial_x + \left(g_y + \frac{f h_x}{h}\right) \partial_y. \label{eqn;NJ2}
\end{gather}
Note that since this is dimension two, $\big(P^\perp\big)^*\!\omega$ is non-degenerate at~$p$ if and only if $\big(P^\perp\big)^*\!\omega(X,JX)\!$ $\neq 0$ for some $X$ at~$p$. We next calculate $\big(P^\perp\big)^*\omega(X,JX)$, and without loss of generality we may assume $X= \partial_x$.

From equation~\eqref{eq;pullomega} of Proposition~\ref{prop;pull} we have
\begin{gather*}
\big(P^\perp\big)^*\omega(X,JX) = - \langle M ( \nabla_{\partial_x} J ) Z, ( \nabla_{J \partial_x} J ) Z \rangle +
 \langle M \partial_x , J\partial_x \rangle,
\end{gather*}
where $Z$ is any vector such that $Z + iJZ$ is unit vector spanning $T^{0,1}$, and $ \langle - , - \rangle$ is the standard round metric, which equals the metric $ds^2$ above at $p$. We choose
\begin{gather*}
Z = c \partial_x, \qquad \textrm{where} \qquad c = \frac{1}{\sqrt{1+f^2+g^2}}
\end{gather*}
so that $c(\partial_x + i (f \partial_x + g \partial_y ) )$ is a unit vector.

We have $P^\perp(Y) = \langle Y, Z+iJZ \rangle (Z+iJZ)$, and $M= \operatorname{Im} P^\perp$, so that
\begin{gather*}
P^\perp(\partial_x) = \big( c^2 \partial_x + c^2 f ( f \partial_x + g \partial_y ) \big) + i c^2 g \partial_y ,
\end{gather*}
and
\begin{gather*}
M(\partial_x) = c^2 g \partial_y \qquad \textrm{and} \qquad M(\partial_y) = -c^2 g \partial_x.
\end{gather*}
Next we have $ (\nabla J) Z = c (\nabla J) \partial_x $, so using $h(0,0) =1$ and $h_x(0,0) = h_y(0,0) =0$, we have from equation~\eqref{eqn;NJ1} that
\begin{gather*}
M( \nabla_{\partial_x} J) \partial_x = M ( f_x \partial_x + g_x \partial_y ) = c^2 g ( f_x \partial_y - g_x \partial_x ).
\end{gather*}
Since $\big(P^\perp\big)^*\omega$ is alternating and $J \partial_x = f \partial_x + g \partial_y$, we have from equation~\eqref{eqn;NJ2}
\begin{gather*}
\big(P^\perp\big)^*\omega(X,JX) =-c^2 g \langle M ( \nabla_{\partial_x} J) \partial_x , ( \nabla_{\partial_y} J ) \partial_x \rangle +
 c^2 g^2 \langle \partial_y , \partial_y \rangle \\
\hphantom{\big(P^\perp\big)^*\omega(X,JX)}{} =-c^4 g^2 \langle f_x \partial_y - g_x \partial_x , f_y \partial_x + g_y \partial_y \rangle + c^2 g^2 = -c^4 g^2 (f_x g_y - g_x f_y )
 + c^2 g^2.
 \end{gather*}
Since $g$ is nowhere zero, this is non-zero if and only if
\begin{gather*}
f_x g_y - g_x f_y \neq 1+f^2+g^2, \qquad \text{i.e.}, \qquad
\det(dF) \neq \|\partial_x + i J \partial_x \|^2,
\end{gather*}
where $F(x,y) = (f(x,y),g(x,y)) = \tilde J (\partial_x)$.
\end{proof}

Thus, on $S^2$, the condition that the pullback form $\big(P^\perp\big)^*\omega$ is degenerate is quite rare, though one can construct almost complex structures on $S^2$ such that equation~\eqref{eq;bound} of Corollary~\ref{cor;bound} fails to hold.

By a similar local calculation for $\R^6$, we see that the pullback form $\big(P^\perp\big)^*\omega$ of an arbitrary~$J$ on~$S^6$ can be degenerate at a point, even if~$J$ is integrable near that point. In fact, one can consider an integrable $J$ in a neighborhood of the origin in~$\R^6$ which splits into the direct sum of an almost complex structure $K$ on $\R^2$ of the form the proof of Proposition~\ref{prop;S2calc}, together with two standard almost complex structures on~$\R^2$, making this $J$ orthogonal at the origin of $\R^6$ (only). Then the same calculation shows the transport of this local~$J$ to~$S^6$ (extended smoothly to the whole sphere) satisfies that $\big(P^\perp\big)^*\omega \in \Om^2\big(S^6\big)$ is degenerate if~$K$ is degenerate.

We next show that the bound in~\eqref{eq;halfbound} in Corollary \ref{cor;bound} allows us to conclude considerably more than one would a priori know for an arbitrary manifold. To see this, we first consider the left hand side of the inequality in~\eqref{eq;halfbound} as a quadratic form on each tangent space.

\begin{defn}Let $(M,J, \langle - , - \rangle)$ be an almost complex Riemmannian manifold, with associated Levi-Civita connection~$\nabla$ and tangent algebra $m$. Let $( - , - )$ be the induced metric from Lemma~\ref{lem;JRM}.

For any $Z \in T_x M$ define
\begin{gather*}
Q_Z(X) = \big( (\nabla_X J) Z, J m(Z,X) \big) .
 \end{gather*}
 Note that $Q_{JZ}(X) = Q_Z(X) $ for all $Z$ and $X$.
 \end{defn}

Note that this quadratic form vanishes identically whenever $J$ is strongly integrable with respect to $\nabla$.
Next we show that the condition that $Q_Z(X)$ is negative definite, even at any one point, already guarantees $J$ to be non-integrable on any manifold.

 \begin{prop} Let $(M,J, \langle - , - \rangle)$ be as above, and let $Q_Z$ be the associated quadratic form
 for any fixed $Z$. If there are any two non-zero tangent vectors $X$ and $Z$ at any point of $M$ for which $Q_Z(X) < 0$ and $Q_Z(JX) < 0$, then $J$ is not integrable.
 \end{prop}

\begin{proof} We show that if $J$ is integrable then, the form
 \begin{gather*}
 K(X) = Q_Z(X)+ Q_Z(JX)
 \end{gather*}
 is positive semi-definite, so that for each $X$,
 \begin{gather*}
 Q_Z(X) \geq 0 \qquad \textrm{or} \qquad Q_Z(JX) \geq 0.
 \end{gather*}

Since $m$ is $J$-anti-linear and $J$ is orthogonal with respect to the metric $( - , - )$ from Lemma~\ref{lem;JRM}, we have
\begin{gather*}
Q_Z(JX) + Q_Z(X) = \big( (\nabla_{JX} J) Z, J m(Z,JX) \big) + \big( (\nabla_X J) Z, J m(Z,X) \big) \\
\hphantom{Q_Z(JX) + Q_Z(X)}{} = \big( (\nabla_{JX} J) Z, m(Z,X) \big) - \big( J (\nabla_X J) Z, m(Z,X) \big)\\
\hphantom{Q_Z(JX) + Q_Z(X)}{}= \big( m(X,Z) , m(Z,X) \big).
\end{gather*}
Recall, $J$ is integrable if and only if $m(X,Z) = m(Z,X)$ for all $X$ and $Z$, so in this case
\begin{gather*}
Q_Z(JX) + Q_Z(X) = \| m(X,Z) \|^2 \geq 0. \tag*{\qed}
\end{gather*} \renewcommand{\qed}{}
\end{proof}

Thus, what is gained from Corollary \ref{cor;bound} when~\eqref{eq;halfbound} holds on $S^6$, beyond what is a priori true for all manifolds, is the case that the left hand side of~\eqref{eq;halfbound} is nowhere negative definite.

 Finally, we mention that in \cite{BHL} it was shown that there is an open set, containing all orthogonal almost complex structures, none of which are integrable. It would be interesting to see how this open set compares to the set of almost complex structures which are forbidden to be integrable by Corollary~\ref{cor;bound}.

\subsection*{Acknowledgements}
 I gratefully acknowledge the Queens College sabbatical/fellowship leave program, which provided me with time to conduct some of this research. I~thank Arthur Parzygnat for comments on a~preliminary version of this paper, and also thank the referees for their suggestions, which have improved this paper.

\pdfbookmark[1]{References}{ref}
\LastPageEnding

\end{document}